\documentclass[12pt, reqno]{amsart}%
\usepackage{amsmath, amsthm, amscd, amsfonts, amssymb, graphicx, color}
\usepackage[bookmarksnumbered, colorlinks, plainpages]{hyperref}
\usepackage{amsmath}
\usepackage{amsfonts}
\usepackage{amssymb}
\usepackage{graphicx}%
\usepackage{multicol}
\providecommand{\U}[1]{\protect\rule{.1in}{.1in}}
\makeatletter
\@namedef{subjclassname@2020}{	\textup{2020} Mathematics Subject Classification}
\makeatother
\newtheorem{theorem}{Theorem}[section]

\newtheorem{proposition}[theorem]{Proposition}
\newtheorem{corollary}[theorem]{Corollary}
\theoremstyle{definition}
\newtheorem{definition}[theorem]{Definition}

\newtheorem{remark}[theorem]{Remark}
\numberwithin{equation}{section}

\newcommand{\be}{\begin{equation}}
\newcommand{\ee}{\end{equation}}
\begin{document}
\title[On Some Intersection Properties of Finite Groups]{On Some Intersection Properties of Finite Groups}
\author[Das]{Angsuman Das$^{\flat}$}
\email{\textcolor[rgb]{0.00,0.00,0.84}{angsuman.maths@presiuniv.ac.in}}
\address{Department of Mathematics\\ Presidency University, Kolkata, India} 

\author[Mandal]{Arnab Mandal}
\email{\textcolor[rgb]{0.00,0.00,0.84}{arnab.maths@presiuniv.ac.in}}
\address{Department of Mathematics\\ Presidency University, Kolkata, India} 

\thanks{$^\flat$Corresponding author}

\begin{abstract}
In this article, we introduce the study of a class of finite groups $G$ which admits a subgroup which intersects all non-trivial subgroups of $G$.  We also explore a subclass of it consisting of all groups $G$ in which the prime order elements commute. In particular, we discuss the relationship between these class of groups with other known classes of finite groups, like simple groups, perfect groups etc. Moreover, we also prove some results on the possible orders of such groups. Finally, we conclude with some open issues.
\end{abstract}
\maketitle
\setcounter{page}{1}

\section{Introduction}
Various interesting results in group theory like Lagrange's theorem, Sylow theorems, Feit-Thompson theorem etc. hold only for finite groups, i.e., there is a close relationship between finite combinatorics and finite group theory. This even gives rise to various peculiar finite groups like Heisenberg groups modulo an odd prime $p$, Quaternion groups, Elliptic curve groups over finite fields etc. In this article, we study and generalize some remarkable properties of one such group, namely the generalized Quaternion group $Q_{2^n}=\langle a,b: \circ(a)=2^{n-1}; a^{2^{n-2}}=b^2; ba=a^{-1}b \rangle$, $n\geq 3$.  It is known that any non-abelian $2$-group with a unique subgroup of order $2$ is isomorphic to $Q_{2^n}$ (See Theorem 5.4.10.ii, p.199 \cite{gorenstein}). Moreover, this unique subgroup has a property that any other non-trivial subgroup of $Q_{2^n}$ intersects it non-trivially. We generalize this property for any arbitrary group $G$.

\begin{definition}
    A group $G$ is said to have subgroup intersection property or $SIP$ if there exists a proper subgroup $H$ of $G$ which intersects all non-trivial subgroups $X$ of $G$ non-trivially, i.e., $|X\cap H|>1$.
\end{definition}

We also define a subclass of $SIP$ groups as follows:

\begin{definition}
    A group $G$ is said to have strong subgroup intersection property or $SSIP$ if there exists a unique proper subgroup $H$ of $G$ which intersects all non-trivial subgroups of $G$ non-trivially.
\end{definition}

The class of $SSIP$ groups is a non-empty proper subclass of $SIP$ groups as $Q_{2^n}$ is a $SIP$ group which is not $SSIP$, and for any odd prime $p$, $\mathbb{Z}_{p^2}\rtimes \mathbb{Z}_p$ is an example of $SSIP$ group.

On the other hand, it is known that if $G$ is a finite group such that all elements of prime power order commute, then $G$ is abelian. So, what about the finite groups in which all prime order elements commute? Are they necessarily commutative? The answer is negative and $Q_{2^n}$ serves as a family of counterexamples, as it has a unique element of order $2$. This motivates the definition of another class of groups.

\begin{definition}
    A group $G$ is said to have $POEC$ property if all the elements of prime order in $G$ commute.
\end{definition}

As abelian groups are always $POEC$, we consider only non-abelian $POEC$ groups. As we will show later, the class of non-abelian $POEC$ groups also form a subclass of $SIP$ groups. The inter-relationship between the classes of $SIP$, $SSIP$ and $POEC$ groups, as shown in Figure \ref{picture}, is the main topic of discussion of the current article. Although the definitions of these classes allow the group to be infinite, in what follows, we assume $G$ to be a finite group.

\begin{figure}[ht]
	\centering
	\begin{center}
        \includegraphics[scale=0.5]{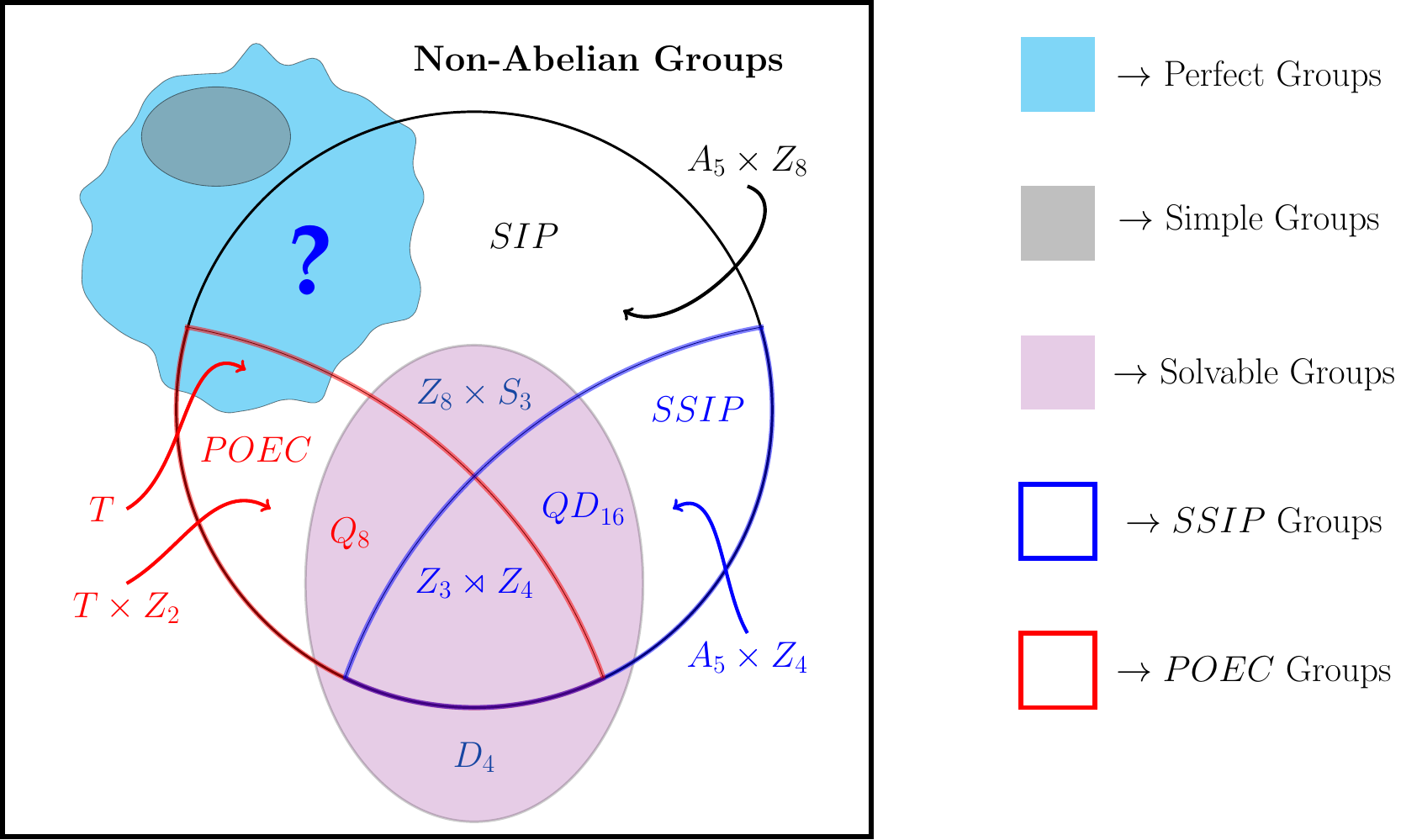}
		\caption{Inter-relationship between various types of finite non-abelian groups}
		\label{picture}
	\end{center}
\end{figure}

\subsection{Preliminaries and Basic Results}
Before going to elaborate results, we define some terminologies and their relations with each other. In this paper, by $S_p$, we mean Sylow $p$-subgroup of the group in context, and not the symmetric group on $p$ symbols. 
\begin{definition}
    Let $G$ be a finite group, $\pi(G)$ be the set of primes dividing $|G|$ and $p\in \pi(G)$. Define $P[G]$ to be the subgroup generated by all elements of prime order in $G$, i.e., $$P[G]=\langle \{x \in G: \circ(x) \mbox{ is prime}\} \rangle.$$
    Also define $G_p$ be the subgroup generated by all elements of order $p$ in $G$, i.e., $$G_p=\langle \{x \in G: \circ(x)=p\} \rangle.$$
\end{definition}

Clearly, $G_p \leq P[G]$ and both are characteristic subgroups of $G$ and hence normal in $G$. Moreover, $P[G]$ intersects all non-trivial subgroups of $G$ non-trivially, i.e., if $X$ is a non-trivial subgroup of $G$, then $|X\cap P[G]|>1$.  One can easily observe that a finite group $G$ is $SIP$ if and only if $P[G]$ is a proper subgroup of $G$. It is to be noted that for an arbitrary group $G$, $P[G]$ and $G_p$ may be equal to the entire group $G$.

It can be shown that if $G$ is a non-abelian $POEC$ group, then $P[G]$ is a proper abelian subgroup of $G$ and hence $G$ is not simple.  This also shows that a non-abelian $POEC$ group is also a $SIP$ group.

We now recall a few well known theorems on finite groups which will be used in what follows.

\begin{theorem}[Schur-Zassenhaus Theorem]
    Any normal Hall subgroup $K$ of a finite group $G$ possesses a complement, that is, there is some subgroup $L$ of $G$ such that $KL = G$ and $K\cap L=\{e\}$ (so $G$ is a semidirect product of $K$ and $L$).
\end{theorem}
    
\begin{proposition}[Theorem 10.1.4 \cite{robinson-book}]\label{key-result}
    If $G$ is a finite perfect group and $[G:Z(G)]=t$, then $g^t=e$ for all $g \in G$.
\end{proposition}

\begin{proposition}[Theorem 6.17 \cite{suzuki-book}]\label{sl-perfect}
    Let $p$ be an odd prime. The only non-trivial perfect subgroups of $SL(2,p)$ are: \begin{itemize}
        \item $SL(2,p)$ itself, and 
        \item $SL(2,5)\cong 2.A_5$.
    \end{itemize}
    The second case occurs only if $p\equiv \pm 1~(mod~10)$.
\end{proposition}

\subsection{Organisation of the paper}
As discussed earlier, for finite non-abelian groups, both $POEC$ and $SSIP$ groups are sub-classes of $SIP$ groups. Since $POEC$ group itself forms an important class of groups, we start with $POEC$ groups in Section \ref{poec-section}, followed by $SIP$ and $SSIP$ groups in Sections \ref{sip-section} and \ref{ssip-section} respectively. Finally, we conclude with some open issues in Section \ref{conclusion}. 

\section{POEC Groups}\label{poec-section}
Let $G$ be a non-abelian $POEC$ group. Then all the elements of $P[G]$ are of square free order. Thus, if $|G|=p^{\alpha_1}_1p^{\alpha_2}_2\cdots p^{\alpha_k}_k$, then $$P[G]\cong G_{p_1}\times G_{p_2}\times \cdots \times G_{p_k} \cong \mathbb{Z}^{\beta_1}_{p_1}\times \mathbb{Z}^{\beta_2}_{p_2}\times \cdot \times \mathbb{Z}^{\beta_k}_{p_k},$$ where $\mathbb{Z}^{\beta_i}_{p_i}$ is the direct product of $\beta_i$ copies of $\mathbb{Z}_{p_i}$, $\beta_i\leq \alpha_i$ for $i=1,2\ldots,k$ and not all $\alpha_i=\beta_i$. Note that for $POEC$ groups, $$G_p=\langle \{x \in G: \circ(x)=p\} \rangle = \{x \in G: \circ(x)=p\}.$$

\begin{theorem}
Let $G$ be a finite group. $G$ is a $POEC$ group if and only if all elements of square-free order forms an abelian subgroup of $G$.
\end{theorem}

\begin{proof}
If $G$ is a $POEC$ group, then the theorem holds using the above discussion. Conversely, let $H$ be set of all elements of square-free order which forms an abelian subgroup of $G$. Clearly, all elements of prime order are in $H$ and $H$ is abelian. Thus $G$ is a $POEC$ group. 
\end{proof}

\begin{proposition}\label{bunch}
    Let $G$ be a finite $POEC$ group. Then the following are true:
    \begin{enumerate}
        \item(POEC is subgroup-closed and direct-product closed) If $H\leq G$, then $H$ is a $POEC$ group. If $G_1$ and $G_2$ are two $POEC$ groups, then $G_1\times G_2$ is also a $POEC$ group.
        \item(POEC is not quotient-closed) If $H\lhd G$, then $G/H$ may not be a $POEC$ group.
        \item If $G$ has a normal Sylow-$p$-subgroup $P$, then $G/P$ is a $POEC$ group. (the statement is also true if we replace Sylow subgroup by a Hall subgroup)
        \item If $G$ is non-abelian, then $|G|$ is not square-free.
        \item A positive integer is called {\it almost square-free} if it is divisible by $p^2$ for at most one prime $p$. If $|G|$ is almost square-free, then $G$ is super-solvable.
        \item If $|G|$ is divisible by square of at most two distinct primes, $G$ is solvable.
        \item If $8$ does not divide $|G|$, $G$ is solvable.
        \item If $G$ is not solvable, then there exists two odd primes $p,q$ such that $8p^2q^2$ divides $|G|$.
        \item If $G$ is perfect, then $8p^2q^2$ divides $|G|$.
    \end{enumerate}
\end{proposition}

\begin{proof}
    \begin{enumerate}
        \item The proofs follow from the definition of $POEC$ groups.
        \item Consider the group $G=(\mathbb{Z}_4\times \mathbb{Z}_2)\rtimes \mathbb{Z}_4$ with GAP ID (32,2) which is defined by the following presentation:
$\langle a, b, x \mid a^4 = b^2 = x^4, ab = ba, bx = xb, xax^{-1} = ab\rangle$. It is a $POEC$ group. Now $G$ has a normal subgroup $N$ isomorphic to Klein's $4$-group, such that the quotient $G/N$ is isomorphic to $D_4$, the dihedral group of order $8$, which itself is not a $POEC$ group.
\item Let $|G|=p^{\alpha_1}_1p^{\alpha_2}_2\cdots p^{\alpha_k}_k$ and $|P|=p^{\alpha_1}_1$. As $P$ is normal in $G$, by Schur – Zassenhaus theorem, $P$ has a complement $Q$ in $G$ with $|Q|=p^{\alpha_2}_2\cdots p^{\alpha_k}_k$. As $G=PQ$, any element $g \in G$ can be expressed as $x_1x_2$ where $x_1 \in P$ and $x_2 \in Q$. Thus any element $gP \in G/P$ can also expressed as $xP$ where $x\in Q$.

Let $xP, yP \in G/P$ are prime order elements such that $x,y \in Q$. Let $\circ(xP)=p_2$ and $\circ(yP)=p_3$. Thus $x^{p_2},y^{p_3} \in P$. Also $x^{p_2},y^{p_3} \in Q$. Hence $x^{p_2}=y^{p_3}=e$, i.e., $\circ(x)=p_2$ and $\circ(y)=p_3$. As $G$ is a $POEC$ group, we have $xy=yx$, i.e., $xP\cdot yP=yP\cdot xP$. Thus $G/P$ is also a $POEC$ group.
\item If possible, let $|G|$ be square-free. Then order of all of its elements are also square-free. Hence $SQF(G)=G$. But this implies that $G$ is abelian, a contradiction. Hence the result holds.
\item Suppose, $G$ is not supersolvable. Thus $G$ is not abelian. Thus by previous result, $|G|$ is not square-free, i.e., there exists a prime $p$ such that $p^2$ divides $|G|$. If $p$ is the only such prime, i.e.,  $|G|$ is almost square-free, then $|G|=p^ap_1p_2\cdots p_k$ with $a\geq 2$ and $|SQF(G)|=p^bp_1p_2\cdots p_k$ with $b\leq a$. Consider $K$ the subgroup of $SQF(G)$ of order $p_1p_2\cdots p_k$. Clearly $K$ is cyclic and $K\lhd G$. Now, as $G/K$ is a $p$ group, it is supersolvable. Also $K$ being cyclic, we deduce that $G$ is supersolvable, a contradiction. Thus the result holds.
\item The proof is similar to the above proof.
\item If $|G|$ is $m$ or $2m$, where $m$ is odd, then $G$ is solvable. So we assume that $|G|=4m$, where $m$ is odd. Let $H_2$ be the subgroup generated by elements of order $2$ in $G$. Then $H_2$ is normal in $G$ and $|G/H_2|=m$ or $2m$. In any case, $G/H_2$ is solvable. Moreover, as $H_2$ is abelian, and hence solvable. Thus $G$ is solvable.
\item It follows from the above two results.
\item Since $G$ is perfect, it is not solvable and it follows from the above result.
    \end{enumerate}
\end{proof}


\begin{proposition}
 A nilpotent group is $POEC$ if and only if all of its Sylow subgroups are $POEC$.   
\end{proposition}

 \begin{proof}
     Let $G$ be a nilpotent $POEC$ group. As $POEC$ is a subgroup-closed property, any subgroup of $G$ and in particular Sylow subgroups are $POEC$. Conversely, let $G$ be a nilpotent group such that its Sylow subgroups $P_i$'s are $POEC$. As $POEC$ is direct product closed, $G$ is a $POEC$ group.
 \end{proof}

 The above proposition suggests that we should try to explore $POEC$ $p$-groups. Our focus is on $p$-groups of order $p^n$, where $n\geq 3$, as $p$-groups are commutative for $n\leq 2$. It is an interesting fact that, for an odd prime $p$, there is a non-abelian $POEC$ group of order $p^3$, namely $\mathbb{Z}_{p^2}\rtimes \mathbb{Z}_p$.  Using this, we can always construct a non-abelian $POEC$ group of order $p^n$ for all $n\geq 3$, as $(\mathbb{Z}_{p^2}\rtimes \mathbb{Z}_p)\times \mathbb{Z}_{p^{n-3}}$ serves our purpose. Moreover, this result also holds for $2$-groups due to $Q_{2^n}$ for $n\geq 3$. Note that for a non-abelian $POEC$ $p$-group ($p$ is odd) of order $p^n$, by Theorem 5.4.10.ii, p.199 \cite{gorenstein}, we have $p<|G_p|<p^n$. In fact for all $r$ with $2 \leq r \leq n-1$, we can construct a non-abelian $POEC$ $p$-group of order $p^n$ such that $|G_p|=p^r$. For $p=2$, along with the above values of $r$, we can also get $|G_p|=2$. On the other hand, we would like to mention that there always exists a non-$POEC$ group of order $p^n$ for all $n\geq 3$ due to existence of Heisenberg group and Dihedral groups.

 \subsection{Center of $POEC$ groups}
 
It is observed via numerical examples that $POEC$ groups have non-trivial center. In this section, we prove some partial results in this direction.

\begin{proposition}\label{[G:P[G]]is_prime_power}
    If $G$ is a $POEC$ group such that $[G:P[G]]$ is a prime power, then $Z(G)$ is non-trivial.
\end{proposition}
\begin{proof}
    Let  $|G|=p^{\alpha_1}_1p^{\alpha_2}_2\cdots p^{\alpha_k}_k$ where $p_i$'s are distinct primes and without loss of generality, let $[G:P[G]]=p^{\beta_1}_1$ where $\beta_1\leq \alpha_1$. Thus $G_{p_i}=S_{p_i}$ for $i=2,\ldots,k$ where $S_{p_i}$ denotes Sylow $p_i$-subgroup of $G$. Observe that $H=S_{p_2}S_{p_3}\cdots S_{p_k}$ is a normal Hall subgroup of $G$ and $G=HS_{p_1}$. Choose $a \in Z(S_{p_1})$ such that $\circ(a)=p_1$. Note that any element $g \in G$ is of the form $g=g_{p_2}g_{p_3}\cdots g_{p_k}\cdot b$, where $g_{p_i}\in G_{p_i}$ for all $i\geq 2$ and $b \in S_{p_1}$. Then 
    $$\begin{array}{lll}
      a\cdot g   & =a\cdot (g_{p_2}g_{p_3}\cdots g_{p_k})\cdot b & \\
        & =((g_{p_2}g_{p_3}\cdots g_{p_k})\cdot a) \cdot b, & \mbox{ as } G\mbox{ is POEC}\\
        & =(g_{p_2}g_{p_3}\cdots g_{p_k})\cdot (b \cdot a), & \mbox{ as } a \in Z(S_{p_1})\\
         & =g \cdot a & \\
    \end{array}$$
Thus $a \in Z(G)$. 
\end{proof}

\begin{corollary}
    If $G$ is a $POEC$ group such that $|G|$ is almost square-free, i.e., $|G|=p^{\alpha_1}_1p_2\cdots p_k$, then $Z(G)$ is non-trivial.
\end{corollary}

\begin{remark}
    The above corollary is not true in general, as $A_4$ has  trivial center. 
\end{remark}

\begin{proposition}\label{non-trivial-center-second-result}
    If $G$ is a $POEC$ group such that $|G|=p_1p^{\alpha_2}_2\cdots p^{\alpha_k}_k$ where $p_i\nmid (p_1-1)$ for all $i$, then $Z(G)$ is non-trivial.
\end{proposition}
\begin{proof}
    Using $N/C$-theorem on $G_{p_1}$, we get $$G/C_G(G_{p_1})\leq Aut(\mathbb{Z}_{p_1}).$$ Note that $|Aut(\mathbb{Z}_{p_1})|=p_1-1$ and as $p_i\nmid (p_1-1)$ for all $i$, we have $G=C_G(G_{p_1})$, i.e., $G_{p_1}\subseteq Z(G)$.
\end{proof}

\begin{corollary}
    If $G$ is a $POEC$ group such that $|G|=p_1p^{\alpha_2}_2\cdots p^{\alpha_k}_k$ where $p_1$ is the smallest prime factor of $|G|$, then $G_{p_1}\subseteq Z(G)$.
\end{corollary}

\begin{remark}
    If $G$ is a $POEC$ group such that $|G|=p^{\alpha_1}_1p^{\alpha_2}_2\cdots p^{\alpha_k}_k$ where $\alpha_i\geq 1$, then $Z(G)$ is non-trivial if and only if $G_{p_j}\cap Z(G)$ is non-trivial for some $j$. The above $p_j$ is not necessarily the smallest prime divisor of $|G|$, e.g., $(\mathbb{Z}_2\times \mathbb{Z}_2)\rtimes \mathbb{Z}_9$ is a $POEC$ group with a center of order $3$.
\end{remark}

\subsection{Perfect $POEC$ groups}
 
\begin{theorem}
    The smallest non-solvable group with $POEC$ must be perfect.
\end{theorem}
\begin{proof}
    Let $G$ be the smallest non-solvable group with $POEC$. If $G'\subsetneq G$, then $G'$ is a proper subgroup of $G$ and hence $G'$ is a $POEC$ group and hence solvable. Also as $G/G'$ is abelian, it is solvable. Thus we must have $G$ to be solvable. Thus $G'=G$, i.e., $G$ is perfect.
\end{proof}

\begin{remark}\label{perfect-T}
    There is a perfect $POEC$ group $T$ of order $1215000=2^3\cdot 3^5\cdot 5^4$. Using GAP \cite{GAP4}, one can check that it is the smallest non-solvable $POEC$ group. We will denote this group by $T$ throughout the paper.\footnote{The authors are grateful to Professor Alexander Hulpke for pointing out this example.} Since direct product of perfect groups are perfect, there exist infinitely many perfect $POEC$ groups.
\end{remark}

\begin{remark}
    There exist finite $POEC$ groups which are neither solvable nor perfect. Let $P$ be a perfect $POEC$ group and $A$ be any abelian group. Set $G=P\times A$. Then $G$ is a $POEC$ group. As $P\leq G$, $G$ is non-solvable and as $G'\cong P'\times A'\cong P < G$,  $G$ is not perfect.
\end{remark}

\begin{remark}
The smallest order of a $POEC$ group which is neither solvable nor perfect is $2\times 1215000=2430000$. Clearly $T\times \mathbb{Z}_2$ is a valid candidate of that order. Let $G$ be any such group. As $G$ is non-solvable and non-perfect, it must have a proper perfect $POEC$ subgroup. As $T$ is the smallest perfect $POEC$ group, $G$ must have a proper subgroup at least as large as $|T|$. Thus $|G|\geq 2|T|$.
\end{remark}

\begin{theorem}\label{cyclic-sylow-implies-not-perfect}
    If  $G$ is a $POEC$ group with a cyclic Sylow-$p$-subgroup, then $G$ is not perfect.
\end{theorem}

\begin{proof}
    If possible, let $G$ be perfect. Let $S_p=\langle a \rangle$ be the cyclic Sylow-$p$-subgroup of $G$ of order $p^k$ and $G_p$ be the subgroup generated by elements of order $p$ in $G$. Then $\mathbb{Z}_p\cong G_p\leq S_p$ and $G_p$ is normal in $G$. Then by $N/C$ - theorem, $G/C_G(G_p)$ is congruent to a subgroup of $\mathbb{Z}^*_p$. As $\mathbb{Z}^*_p$ is cyclic, $G/C_G(G_p)$ is abelian. Again as $G$ is perfect, $G/C_G(G_p)$ is also perfect. Thus $G/C_G(G_p)$ is trivial, i.e., $G=C_G(G_p)$, i.e., $G_p\leq Z(G)$.

    Let $[G:Z(G)]=t$. Then $p^k$ does not divide $t$ and hence $a^t\neq e$. But by Proposition \ref{key-result}, $g^t=e$ for all $g \in G$, a contradiction. Thus $G$ is not perfect.
\end{proof}

\begin{corollary}\label{alpha>1}
    If $G$ is a perfect $POEC$ group with $|G|=p^{\alpha_1}_1p^{\alpha_2}_2\cdots p^{\alpha_k}_k$, then $\alpha_i>1$ for all $i$.
\end{corollary}

\begin{proof}
If $\alpha_i=1$ for some $i$, then the corresponding Sylow-$p_i$-subgroup is cyclic and hence by the above theorem, $G$ is not perfect, a contradiction.
\end{proof}

    


\begin{proposition}\label{S2=Q8}
    If $G$ is a perfect $POEC$ group such that $2^4$ does not divide $|G|$, then the Sylow-$2$-subgroup of $G$ is isomorphic to the quaternion group, $Q_8$ and $2\mid |Z(G)|$. 
\end{proposition}
\begin{proof}
By Proposition \ref{bunch}(9) and as $16$ does not divide $|G|$, we have $|G|=8p^{\alpha_1}_1p^{\alpha_2}_2\cdots p^{\alpha_k}_k$. Again as $G/G_2$ is perfect, i.e., $4$ divides $|G/G_2|$, we have $|G_2|=2$. As $G_2$ is the subgroup generated by elements of order $2$ in $G$, we conclude that the Sylow-$2$-subgroup $S_2$ of $G$ has exactly one element of order $2$, thereby enforcing $G_2\leq Z(G)$. Again, since $S_2$ is a group of order $8$, the only possibilities are $S_2 \cong \mathbb{Z}_8$ or $Q_8$. As $\mathbb{Z}_8$ is cyclic, by Theorem \ref{cyclic-sylow-implies-not-perfect}, $S_2\cong Q_8$.
\end{proof}

From Proposition \ref{bunch}(9), it is known that if $G$ is a perfect $POEC$ group, then there exist two odd primes $p$ and $q$ such that $8p^2q^2$ divides $|G|$. Now, we are in a position to say something more.
\begin{theorem}\label{perfect-8p3q3}
     If $G$ is a perfect $POEC$ group, then $8p^3q^3$ divides $|G|$.
\end{theorem}
\begin{proof} We prove the result by contradiction. If the theorem does not hold, then there exists a perfect $POEC$ group $G$ such that $|G|=2^\alpha p^{\beta}_1p^2_2 \cdots p^2_t$, where $\alpha\geq 3,\beta\geq 2,t\geq 2$ and $p_i$'s are distinct odd primes. This follows from Proposition \ref{bunch}(10) and Corollary \ref{alpha>1}. 

Observe that $S_{p_i}=G_{p_i}\cong \mathbb{Z}_{p_i}\times \mathbb{Z}_{p_i}$ for all $i\geq 2$, using Theorem \ref{cyclic-sylow-implies-not-perfect}. Consider the normal subgroup $H=G_{p_2}G_{p_3}\cdots G_{p_t}$ of $G$. Then, as $|G/H|=2^\alpha p^{\beta}_1$, $G/H$ is solvable and being a quotient of a perfect group $G/H$ is perfect, a contradiction.
\end{proof}

\begin{corollary}
    If $G$ is a non-solvable $POEC$ group, then $8p^3q^3$ divides $|G|$.
\end{corollary}

\begin{remark}
    It follows from  Corollary \ref{alpha>1}, that if $G$ is a perfect $POEC$ group and $p$ is a prime dividing $|G|$, then $p^2$ divides $|G|$. In light of Theorem \ref{perfect-8p3q3}, it is natural to ask the following question: If $G$ is a perfect $POEC$ group and $p$ is a prime dividing $|G|$, is it necessary that $p^3$ divides $|G|$? We provide a partial answer to this.
\end{remark}

\begin{theorem}\label{p3|G}
    If $G$ is a perfect $POEC$ group such that $p\mid |G|$ and one of the following conditions hold:
    \begin{itemize}
        \item $2^4$ or $3$ or $5$ does not divide $|G|$;
        \item $p \not\equiv \pm 1~(mod~10)$,
    \end{itemize}
     then $p^3\mid |G|$.
\end{theorem}
\begin{proof}
    As $p\mid |G|$, it follows from  Corollary \ref{alpha>1} that $p^2$ divides $|G|$. Suppose $p^3\nmid |G|$. Then we have $S_p=G_p\cong \mathbb{Z}_p \times \mathbb{Z}_p$. 
    
    Now, by N/C theorem, $G/C_G(G_p)$ is isomorphic to a subgroup of $Aut(G_p)\cong GL(2,p)$. In fact, it is isomorphic to a subgroup of $SL(2,p)$. As $G/C_G(G_p)$ is perfect, $p\nmid G/C_G(G_p)$ and $p\mid |SL(2,p)|$, from Proposition \ref{sl-perfect}, it follows that either $G/C_G(G_p)$ is trivial or $G/C_G(G_p)\cong SL(2,5)$.

    In the later case, we have $|G/C_G(G_p)|=120=2^3\cdot 3 \cdot 5$. From Proposition \ref{S2=Q8}, this can not hold, if one of the above conditions is true. Thus $G/C_G(G_p)$ is trivial, i.e., $G_p\cong \mathbb{Z}_p \times \mathbb{Z}_p \leq Z(G)$, i.e., $p\nmid [G:Z(G)]$. Since $G$ contains elements of order $p$, by Proposition \ref{key-result}, we get a contradiction. Thus $p^3\mid |G|$.
\end{proof}

It is known that if $G$ is a finite perfect group, then $4$ divides $|G|$ and if $8$ does not divide $|G|$, then $3$ does. A stronger result holds for perfect $POEC$ groups.

\begin{theorem}\label{3^4|G}
    If $G$ is a perfect $POEC$ group such that $2^4 \nmid |G|$, then $3^4 \mid |G|$.
\end{theorem}
\begin{proof}
    By Theorem \ref{p3|G}, it follows that $3^3\mid |G|$. Suppose $3^4\nmid |G|$. Consider the Sylow $3$-subgroup $S_3$ of $G$ of order $27$. Then $S_3$ must be isomorphic to either $\mathbb{Z}_9 \times \mathbb{Z}_3$ or $\mathbb{Z}_9 \rtimes \mathbb{Z}_3$, because the other possibilities of $S_3$, namely $\mathbb{Z}_{27}, \mathbb{Z}_3 \times \mathbb{Z}_3\times \mathbb{Z}_3$ and $Heis(\mathbb{Z}_3)$, can be ruled out. Thus, in both the cases, $G_3 \cong \mathbb{Z}_3 \times \mathbb{Z}_3$.

    Now, by N/C theorem, $G/C_G(G_3)$ is isomorphic to a subgroup of $Aut(G_3)\cong GL(2,3)$. If $C_G(G_3)$ is a proper subgroup of $G$, then as $G$ is perfect, $G/C_G(G_3)$ is also perfect, but $GL(2,3)$, being a group of order $48$, has no perfect subgroup. Thus $G=C_G(G_3)$, i.e., $G_3 \leq Z(G)$, i.e., $9$ divides $|Z(G)|$.

    Let $[G:Z(G)]=t$. Clearly $3$ divides $t$ and $9$ does not divide $t$. By Proposition \ref{key-result}, $g^t=e$ for all $g \in G$. In particular if $g$ belongs to a Sylow-$3$-subgroup $S_3$ of $G$, this implies that $g^3=e$, i.e., every element of $S_3$ is of order $3$. However, as $S_3$ is isomorphic to either $\mathbb{Z}_9 \times \mathbb{Z}_3$ or $\mathbb{Z}_9 \rtimes \mathbb{Z}_3$, we get a contradiction. Thus the theorem holds.   
\end{proof}

\begin{remark}
    It can be shown with a little trick, using the same line of argument as in proof of Theorem \ref{3^4|G}, that if $G$ is a perfect $POEC$ group such that $3\mid |G|$, then $3^4 \mid |G|$. Again, using Theorem \ref{p3|G}, it follows that if $G$ is a perfect $POEC$ group such that $5\mid |G|$, then $5^3 \mid |G|$. However with some more effort it can be shown that if $G$ is a perfect $POEC$ group such that $5\mid |G|$ and $2^4\nmid |G|$, then  $5^4 \mid |G|$.
\end{remark}

\begin{remark}
If $G$ is a perfect $POEC$ group with $|G|=2^{\alpha}p^\beta q^\gamma$, where $p,q$ are distinct odd primes, then by Theorem \ref{perfect-8p3q3}, we conclude that $\alpha,\beta,\gamma\geq 3$. Moreover, if $p,q\not\equiv \pm 1~(mod~10)$, using similar tricks as that in the proof of Theorem \ref{p3|G}, we can show that $\beta,\gamma\geq 4$.
\end{remark}

\section{SIP Groups}\label{sip-section}
In this section ,we discuss about $SIP$ groups and its properties. Clearly all non-abelian $POEC$ groups are $SIP$ groups, but the converse may not be true, e.g., $\mathbb{Z}_4 \times S_3$. As elements of order $2$ and order $3$ does not commute in $S_3$, $\mathbb{Z}_4 \times S_3$ is not $POEC$. However, $\{0,2\}\times S_3$ is a proper subgroup which intersects all proper subgroups of $\mathbb{Z}_4 \times S_3$ non-trivially.

Let $G$ be a finite $SIP$ group and let $\mathcal{S}$ be the collection of all proper subgroups of $G$ which intersects all proper subgroups of $G$ non-trivially. As $G$ is a $SIP$ group, $\mathcal{S}$ is non-empty. Note that $\mathcal{S}$ is closed with respect to taking intersection. So $\mathcal{S}$ has a minimum element.

\begin{proposition}
    If $H$ is a maximal element of $\mathcal{S}$, then $H$ is a maximal subgroup of $G$.
\end{proposition}
\begin{proof}
    Suppose $H$ is not a maximal subgroup of $G$. Then there exists a proper subgroup $K$ of $G$ such that $H \subsetneq K \subsetneq G$. However, this implies that $K \in \mathcal{S}$ and hence $H$ is not a maximal element of $\mathcal{S}$, a contradiction.
\end{proof} 

\begin{proposition}
    Let $G$ be a finite $SIP$ group. Then $G$ is not simple.
\end{proposition}
\begin{proof}
  Let $K$ be the intersection of all elements of $\mathcal{S}$. Then $K$ is a non-trivial subgroup of $G$, as $K$ contains all prime order elements of $G$. In fact, $K$ is the smallest subgroup of $G$ containing all prime order elements of $G$, i.e., $K=P[G]$. Thus $K$ or $P[G]$ is a proper characteristic subgroup of $G$ and hence normal in $G$. Thus $G$ is not simple.  
\end{proof}

It is to be noted that if $G$ is a $SIP$ group, $P[G]$ intersects all subgroups of $G$ non-trivially. Moreover $P[G]$ is the smallest subgroup of $G$ which intersects all subgroups of $G$ non-trivially.

\begin{corollary}
    Let $G$ be a $SIP$ group such that $P[G]$ is abelian, then $G$ is a $POEC$ group.
\end{corollary}

\begin{theorem}
    Any finite nilpotent group $G$ is $SIP$ group if and only if at least one Sylow subgroup of $G$ is $SIP$.
\end{theorem}
\begin{proof}
    Let $G$ be a finite nilpotent group with Sylow subgroups $S_{p_1},S_{p_2},\ldots,S_{p_k}$ such that $S_{p_1}$ is $SIP$. Then $S_{p_1}$ has a proper subgroup $Q_1$ which intersects all subgroups of $S_{p_1}$. Then $G\cong S_{p_1}\times S_{p_2} \times\cdots \times S_{p_k}$ has a subgroup $Q_1\times S_{p_2} \times\cdots \times S_{p_k}$ which intersects all subgroups of $G$.

    Conversely, let $G$ be a finite nilpotent $SIP$ group with Sylow subgroups $S_{p_1},S_{p_2},$ $\ldots,S_{p_k}$ and let $Q=Q_1\times Q_2 \times\cdots \times Q_k$ be a subgroup of $G$ which intersects all subgroup of $G$. If $Q_i$ does not intersect all subgroups of $S_{p_i}$ for all $i$, then $S_{p_i}$ has a subgroup $K_i$ which intersects trivially with $Q_i$. Now consider the subgroup $K=K_1\times K_2 \times\cdots \times K_k$ of $G$. Clearly $K\cap Q$ is trivial, a contradiction. Thus there exist $j$, such that $Q_j$ intersect all subgroups of $S_{p_j}$.
\end{proof}

\begin{theorem}
    Any  finite abelian group $G$ is $SIP$ group if and only if $G$ is not elementary abelian.
\end{theorem}
\begin{proof}
Let $G$ be a finite abelian $SIP$ group. As $G$ is abelian, it is nilpotent. Thus $G\cong S_{p_1}\times S_{p_2} \times\cdots \times S_{p_k}$, $k\geq 1$ where $S_{p_i}$ are Sylow subgroups and $S_{p_1}$ is a $SIP$ group. As $S_{p_1}$ is a $SIP$ group, it can not be  elementary abelian.

Conversely, let $G$ be a finite abelian group which is not elementary abelian. Then without loss of generality, $$G \cong \left( \mathbb{Z}_{{p^{r_1}_1}}\times \mathbb{Z}_{{p^{r_2}_1}} \times \cdots \times \mathbb{Z}_{{p^{r_k}_1}}\right) \times \cdots \times \left( \mathbb{Z}_{{p^{s_1}_l}}\times \mathbb{Z}_{{p^{s_2}_l}} \times \cdots \times \mathbb{Z}_{{p^{s_k}_l}}\right),$$ where either $l\geq 2$ or $r_1\geq 2$. If $l\geq 2$, then $G$ can be expressed as $\mathbb{Z}_{p_1p_2}\times H$. Let $P$ be a subgroup of $\mathbb{Z}_{p_1p_2}$ of order $p_1$. Then $P\times H$ is a subgroup which intersects all subgroups of $G$ non-trivially. If $r_1\geq 2$, then express $G$ as $\mathbb{Z}_{{p^{r_1}_1}}\times H$. Let $P$ be a subgroup of $\mathbb{Z}_{{p^{r_1}_1}}$ of order $p_1$. Then $P\times H$ is a subgroup which intersects all subgroups of $G$ non-trivially. Thus, in any case, $G$ is a $SIP$ group.
\end{proof}

In contrast to $POEC$, the smallest non-solvable $SIP$ group is not perfect, where $\mathbb{Z}_4 \times A_5$ is a candidate. On the other hand, as non-abelian $POEC$ groups are $SIP$, there exists perfect $SIP$ groups. Thus it is natural to ask whether there exists a perfect $SIP$ group which is not $POEC$.

It can be shown that if $G$ is a perfect $SIP$ group, then there exist two distinct odd primes $p,q$ such that $8p^2q^2\mid |G|$. Moreover, if $2^4\nmid |G|$, then $3^2\mid |G|$.
\section{SSIP Groups}\label{ssip-section}
In this section, we discuss the properties of $SSIP$ groups. Note that if $G$ is a $SSIP$ group, then $P[G]$ is the unique proper subgroup which intersects all proper subgroups. Earlier it was shown in case of $SIP$ groups that $P[G]$ is a characteristic subgroup of $G$ and it is the smallest subgroup of $G$ which intersects all proper subgroups non-trivially. In case of $SSIP$, we can say something more.

\begin{theorem}
    Let $G$ be a finite group.  $G$ is a $SSIP$ group if and only if $[G:P[G]]$ is prime.
\end{theorem}
\begin{proof}
Let $G$ be a $SSIP$ group. Then from the uniqueness condition, it follows that $P[G]$ is also a maximal subgroup and $[G:P[G]]$ is a prime.

Conversely, since $P[G]$ is a proper characteristic subgroup of $G$, we have $P[G]\lhd G$. Clearly $P[G]$ is the smallest subgroup of $G$ which intersects all proper subgroups of $G$ non-trivially. Suppose $A$ is another proper subgroup of $G$ which intersects all proper subgroups of $G$ non-trivially. Then we have $P[G] \subsetneq A \subsetneq G$. Thus $[G:P[G]]=[G:A][A:P[G]]$ and this contradicts that $[G:P[G]]$ is prime. Hence $P[G]$ is the unique subgroup which intersects all proper subgroups of $G$ non-trivially, i.e., $G$ is $SSIP$.
\end{proof}

Thus, $G$ is a $SSIP$ group if and only if $[G:P[G]]$ is prime.

\begin{proposition}
    A finite cyclic group $G$ is a $SSIP$ group if and only if $|G|=p^2_1p_2\cdots p_k$, where $p_i$'s are distinct primes.
\end{proposition}
\begin{proof}
    If $G$ is a cyclic group of order $p^2_1p_2\cdots p_k$, then the unique subgroup of order $p_1p_2\cdots p_k$ is the required subgroup which intersects all other proper subgroups non-trivially. 
    
    Conversely, let $G$ be a cyclic $SSIP$ group. Thus $|G|$ is not square-free, and hence there exists a prime $p$ such that $p^2\mid |G|$. If $p^3$ or $q^2$ divides $|G|$, where $q$ is a prime different from $p$, then we get more than one subgroups of $G$ which intersects all non-trivial subgroups of $G$. Thus $|G|$ is of the required form.
\end{proof}

In contrast to $POEC$ groups, $SIP$ and $SSIP$ are not subgroup-closed. For example $G=\mathbb{Z}_4\times S_3$ is a $SSIP$ group with the unique subgroup being $\{0,2\}\times S_3$. However, $S_3$ being a subgroup of $G$ is not even $SIP$. 

Direct product of a $SSIP$ group with a $SIP$ group is not $SSIP$, however direct product of a $SIP$ group with any group is a $SIP$ group.

Similarly, though all abelian groups are $POEC$, this is not true for $SIP$ or $SSIP$ groups, as no group of square-free order is $SIP$. Also $SSIP$ groups may have trivial center, e.g., $\mathbb{Z}_5\rtimes \mathbb{Z}_4$.

We have seen earlier that $POEC$ groups may not be solvable. Similarly, $SSIP$ groups may not be solvable, e.g., $\mathbb{Z}_4 \times A_5$.
However, unlike $POEC$ groups, $SSIP$ groups can not be perfect.

\begin{proposition}
    $SSIP$ groups are never perfect.
\end{proposition}
\begin{proof}
    Let $G$ be a $SSIP$ group. Thus $P[G]$ is a normal subgroup of prime index and hence $G/P[G]$ is an abelian group, which implies that $G'\subseteq P[G]$, i.e., $G$ is not perfect.
\end{proof}

Interestingly, groups which are both $POEC$ as well as $SSIP$ have some nice properties.

\begin{theorem}
    Let $G$ be a $POEC$ group which is also $SSIP$. Then $G$ is metabelian and $G$ has a non-trivial center. Moreover, $G$ is a semidirect product of an abelian normal Hall subgroup $H$ of $G$ and a Sylow subgroup of $G$.
\end{theorem}
\begin{proof}
    If $G$ is abelian, there is nothing to prove. We assume that $G$ is non-abelian. Since $G$ is $POEC$, $P[G]$ is abelian and hence solvable. Again, as $G$ is $SSIP$, $G/P[G]$ is a cyclic group and hence abelian. Thus $G$ is metabelian. Again, as $[G:P[G]]$ is prime, by Proposition \ref{[G:P[G]]is_prime_power}, it follows that $Z(G)$ is non-trivial.

    Let $|G|=p^\alpha q^{\beta_1}_1q^{\beta_2}_2\cdots q^{\beta_k}_k$ and $[G:P[G]]=p$, where $p$ and $q_i$'s are distinct primes. Then $H=G_{q_1}G_{q_2}\cdots G_{q_k}$ is an abelian normal Hall subgroup of $G$ and $G\cong H \rtimes S_p$ where $[S_p:G_p]=p$.
\end{proof}

\begin{remark}
From above theorem, it follows that groups, which are both $POEC$ and $SSIP$, are solvable. However, there exist groups which are both $POEC$ and $SSIP$ but not Lagrangian, and hence not supersolvable, e.g., $(\mathbb{Z}_2\times \mathbb{Z}_2)\rtimes \mathbb{Z}_9$. Also note that $SSIP$ groups may have trivial center, e.g., $\mathbb{Z}_5\rtimes \mathbb{Z}_4$.
\end{remark}

\section{Conclusion and Open Issues}\label{conclusion}
To summarize the results obtained so far, our contribution can be segregated into three aspects:
\begin{itemize}
    \item introducing three classes of groups, $SIP$, $SSIP$ and $POEC$,
    \item exploring their properties and inter-relationships, and 
    \item proving some divisibility conditions on orders of perfect $POEC$ and $SIP$ groups.
\end{itemize}
However, a lot more is yet to be explored in this direction and we conclude with some possible directions and open issues.

\begin{enumerate}
   
    \item In Propositions \ref{[G:P[G]]is_prime_power}, \ref{non-trivial-center-second-result} and \ref{S2=Q8}, we have shown that $POEC$ groups, under certain conditions, admit a non-trivial center. We strongly believe that this holds for all finite $POEC$ groups without imposing further constraints, i.e.,\\ {\bf Open Issue 1:} If $G$ is a finite $POEC$ group, then $|Z(G)|>1$.
    \item In Theorem \ref{p3|G}, it was shown that if a prime $p$ divides the order of a perfect $POEC$ group $G$, then under certain conditions on $p$ and the prime factorization of $|G|$, $p^3$ divides $|G|$. Based on experimental observations, we pose the following question: \\ {\bf Open Issue 2:} If $G$ is a finite $POEC$ group and $p \mid |G|$, then $p^3\mid |G|$.
    \item In Remark \ref{perfect-T}, we mentioned about the smallest perfect $POEC$ group $T$. As it turns out that, it is also the smallest perfect $SIP$ group. On the other hand, $T\times T$ is also a perfect $POEC$ group. It would be nice to know about some more perfect $POEC$ or perfect $SIP$ groups, preferably smaller than $T\times T$. As the current GAP perfect group library contains perfect groups of order upto $2$ million, we pose the following question:
    \\ {\bf Open Issue 3:} Is there any perfect $POEC$ group other than $T$ with order less than $2$ million, or what is the second smallest perfect $POEC$ or perfect $SIP$ group?

\end{enumerate}

\section*{Acknowledgements}
The authors are grateful to Professor Alexander Hulpke from Colorado State University for some fruitful discussion over email and for helping the authors with some computation in GAP. The authors also acknowledge the funding of DST-FIST Sanction no. $SR/FST/MS-I/2019/41$ and DST-SERB-MATRICS Sanction no. $MTR/2022/000020$, Govt. of India.


\begin{thebibliography}{9999}
	

 \bibitem{GAP4} The GAP~Group, \emph{GAP -- Groups, Algorithms, and Programming, Version 4.12.2};  2022,  \url{https://www.gap-system.org}.
 \bibitem{gorenstein} D. Gorenstein, Finite Groups, AMS Chelsea Publishing, 1968.
\bibitem{rotman-book} Rotman J J, An Introduction to the Theory of Finite Groups, 4th Edition, Graduate Text in Mathematics, Springer, 1995. 
\bibitem{robinson-book} D.J.S. Robinson, A Course in the Theory of Groups, 2nd Edition, Graduate Text in Mathematics, Springer, 1996. 
\bibitem{suzuki-book} M. Suzuki, Group theory I, Berlin; New York: Springer-Verlag, 1982.
	
\end{thebibliography}
\end{document}